\documentclass[12pt]{amsart}
\usepackage{amscd}
\def\NZQ{\Bbb}               
\def\NN{{\NZQ N}}

\def\ZZ{{\NZQ Z}}

\def\frk{\frak}               

\def\pp{{\frk p}}

\def\qq{{\frk q}}

\def\mm{{\frk m}}

\def\Phi{{\frk n}}
\def\Phi{{\frk N}}
\def\opn#1#2{\def#1{\operatorname{#2}}} 
\opn\chara{char} \opn\length{\ell} \opn\pd{pd} \opn\rk{\lk}\opn\link{link}
\opn\projdim{proj\,dim} \opn\injdim{inj\,dim} \opn\rank{rank}
\opn\depth{depth} \opn\and{and} \opn\grade{grade}
\opn\height{height} \opn\embdim{emb\,dim} \opn\codim{codal}

\opn\Tr{Tr} \opn\bigrank{big\,rank}
\opn\superheight{superheight}\opn\lcm{lcm}
\opn\trdeg{tr\,deg}%
\opn\reg{reg} \opn\lreg{lreg} \opn\ini{in}
\opn\div{div} \opn\Div{Div} \opn\cl{cl} \opn\Cl{Cl}
\opn\Spec{Spec} \opn\Supp{Supp} \opn\supp{supp} \opn\Sing{Sing}
\opn\Ass{Ass} \opn\Min{Min}
\opn\Ann{Ann} \opn\Rad{Rad} \opn\Soc{Soc}
\opn\Im{Im}
 \opn\Ker{Ker} \opn\Coker{Coker} \opn\Am{Am} \opn \inf{inf}
\opn\Hom{Hom} \opn\Tor{Tor} \opn\Ext{Ext} \opn\End{End} \opn\cd{cd}
\opn\Aut{Aut} \opn\id{id}

\opn\nat{nat}
\opn\pff{pf}
\opn\Pf{Pf} \opn\GL{GL} \opn\SL{SL} \opn\mod{mod} \opn\ord{ord}
\opn\cl{cl} \opn\conv{conv} \opn\ext{ext} \opn\rad{rad} \opn\star{star}
\opn\red{red}
\opn\aff{aff} \opn\con{conv} \opn\relint{relint} \opn\st{st}
\opn\lk{lk} \opn\cn{cn} \opn\core{core} \opn\vol{vol}
\opn\link{link} \opn\star{star}
\opn\gr{gr}

\def\pot#1#2{#1[\kern-0.28ex[#2]\kern-0.28ex]}

\opn\dirlim{\underrightarrow{\lim}}
\opn\inivlim{\underleftarrow{\lim}}
\let\tensor=\otimes
\let\iso=\cong

\let\Dirsum=\bigoplus

\def\Implies{\ifmmode\Longrightarrow \else
     \unskip${}\Longrightarrow{}$\ignorespaces\fi}
\def\implies{\ifmmode\Rightarrow \else
     \unskip${}\Rightarrow{}$\ignorespaces\fi}
\def\iff{\ifmmode\Longleftrightarrow \else
     \unskip${}\Longleftrightarrow{}$\ignorespaces\fi}

\let\:=\colon
\newtheorem{Theorem}{Theorem}[section]

\newtheorem{Corollary}[Theorem]{Corollary}
\newtheorem{Proposition}[Theorem]{Proposition}
\newtheorem{Remark}[Theorem]{Remark}

\newtheorem{Example}[Theorem]{Example}

\newtheorem{Definition}[Theorem]{Definition}

\newtheorem{Question}[Theorem]{Question}
\let\epsilon\varepsilon
\let\phi=\varphi
\let\kappa=\varkappa
\def\qed{\ifhmode\textqed\fi
   \ifmmode\ifinner\quad\qedsymbol\else\dispqed\fi\fi}
\def\textqed{\unskip\nobreak\penalty50
    \hskip2em\hbox{}\nobreak\hfil\qedsymbol
    \parfillskip=0pt \finalhyphendemerits=0}
\def\dispqed{\rlap{\qquad\qedsymbol}}

\opn\dis{dis}
\def\pnt{{\raise0.5mm\hbox{\large\bf.}}}

\begin{document}
\title{Relative Cohen--Macaulayness and relative unmixedness of bigraded modules}

\author{Maryam Jahangiri,  Ahad Rahimi}

\subjclass[2000]{ 13D45, 16W50, 13C14, 13F55.  The first author was in part supported by a grant from IPM
(No. 89130115)
The second author was in part supported by a grant from IPM (No. 89130050)}

\address{Maryam Jahangiri, School of Mathematics and Computer Science, Damghan University of Basic Sciences, damghan, Iran.
 School of Mathematics, Institute for Research in Fundamental Sciences (IPM), P.O. Box: 19395-5746, Tehran, Iran.
}\email{jahangiri@dubs.ac.ir}
 \address{ Ahad Rahimi, Department of Mathematics, Razi University, Kermanshah, Iran.
 School of Mathematics, Institute for Research in Fundamental Sciences
(IPM), P. O. Box: 19395-5746, Tehran, Iran.
}\email{ahad.rahimi@razi.ac.ir}

\maketitle
\address{}

 \maketitle
\begin{abstract}
In this paper we study the finitely generated bigraded modules over a standard bigraded polynomial ring which are relative Cohen--Macaulay or relative unmixed with respect to one of the irrelevant bigraded ideals.  A generalization of Reisner's criterion for Cohen--Macaulay simplicial complexes is considered.
\end{abstract}
\maketitle

\section*{Introduction}
Let $S=K[x_1, \dots, x_m, y_1, \dots, y_n]$ be the standard bigraded polynomial ring over a field $K$
 and bigraded irrelevant ideals $P=(x_1, \dots, x_m)$ and $Q=(y_1, \dots,  y_n)$. Let $M$ be a finitely generated bigraded $S$-module. In \cite{AR2} we call $M$ to be relative Cohen--Macaulay with respect to $Q$ if we have only one nonvanishing local cohomology with respect to $Q$. In other words, $\grade(Q, M)=\cd(Q, M)$ where $\cd(Q, M)$ denote the cohomological dimension of $M$ with respect to $Q$. Our aim in this paper is to investigate more about relative Cohen--Macaulay modules and its related topics like relative unmixedness.  We organize this paper as follows: In Section 1, we first ask the following question:

 Let $M$ be relative Cohen--Macaulay with respect to $P$ and $Q$. Is $M$ itself Cohen--Macaulay? We have a counterexample which shows that the question is not true for dimension 2. Even though, for two given ideals $I$ and $J$ of a local ring $R$ and a finitely generated $R$-module of $M$ which is relative Cohen--Macaulay with respect to $I$ and $J$, the question does not hold. We give some especial cases in which the question holds.

  We call $M$ to be relative unmixed with respect to $Q$ if $\cd(Q, M)=\cd(Q, S/\pp)$ for all $\pp \in \Ass M$. We show that relative Cohen--Macaulay modules with respect to $Q$ are relative unmixed with respect to $Q$. The converse does not hold in general. In the case in which every quotient of $M$ is relative unmixed with respect to $Q$ then it holds. Next we change the above question in the following sense:

  Let  $M$ is relative Cohen--Macaulay with respect to $P$ and relative unmixed with respect to $Q$. Is $M$ itself unmixed? The local version of this question in not the case for dimension 2. We prove that  the question has positive answer in the following bigraded cases: $M$ be a bigraded $S$-module for which   i)   $\cd(P,M)\leq 1$ and $\cd(Q, M)\geq 0$,   ii) $M=M_1\tensor_K M_2$ where $M_1$ is a graded $K[x]$-module and $M_2$ is a graded $K[y]$-module and where $S/(\pp_1+\pp_2)S $ is an integral domain for all $\pp_1\in \Ass M_1$ and $\pp_2 \in \Ass M_2$,  iii) every cyclic submodule of $M$ is pure,  iv)  $M=S/I$ where $I$ is a monomial ideal. We believe that the question has  negative answer for dimension 4. Until now we are not succeed to find such a counterexample. We have this question at the end of this section.

  In Section 2,  we describe explicitly the krull-dimension of the graded components of local cohomology of relative Cohen--Macaulay modules. We show that if $M$ is relative Cohen--Macaulay with respect to $Q$ with $\cd(Q, M)=q$, then  $\dim_S H^q_{Q}(M)=p$ where $p=\cd(P, M)$. More general is true for its  graded components, namely if $f_Q(M)=\cd(Q, M)=q$ and $p+q=\dim M$, then  $\dim_{K[x]} H^q_{Q}(M)_j=p$ for $j\ll 0$ where $f_Q(M)$ is the finiteness dimension of $M$ relative to $Q$.  As a consequence, if $M$ is relative Cohen--Macaulay with respect to $Q$, then $H^q_{Q}(M)$ is an Artinian $S$-module if and only if $q =\dim(M)$. In other words, $H^q_{Q}(M)$ is not Artinian unless the ordinary known case.

  In the following section we consider the hypersurface ring $R=S/fS$ where $f$ is a bihomogeneous element of $S$. We show that the local cohomologies $H^i_Q(R)$ for $i=n, n-1$ where $n\geq 2$ are  never finitely generated. Moreover, $H_{Q}^{n}(R)$ is an Artinian $S$-module  for  $m\leq 1$, and $H_{Q}^{n-1}(R)$ is an Artinian $S$-module if and only if $m=0$.

  In the final section, we let  $\Delta$ be a simplicial complex on $[n+m]$ and  $K[\Delta]=S/I_{\Delta}$ its Stanley-Reisner ring. We say that $\Delta$ is relative Cohen--Macaulay with respect to $Q$ over $K$ if $K[\Delta]$ is relative Cohen--Macaulay with respect to $Q$. We show that  $\cd(Q, K[\Delta])=\dim \Delta_W+1$ where $\Delta_W$ is the subcomplex of $\Delta$ whose faces are subsets of $W$. This generalizes the known fact that for every simplicial complex $\Delta$ one has  $\dim K[\Delta]=\dim \Delta +1$. Using this fact and the generalization Hochster's formula \cite{AR2} we prove the following: $\Delta$ is relative Cohen--Macaulay with respect to $Q$ with $\cd(Q,K[\Delta])=q$  if and only if  $\widetilde{H}_{i}((\link F \cup G)_W;K)=0$ for all $F\in \Delta_W$, $G\subset V$ and all $i< \dim \link_{\Delta_W} F$. This in particular implies the Reisner's criterion for Cohen--Macaulay simplicial complexes. A general version of this statement for monomial case is obtained.


\address{}

\address{}

\section{Cohen--macaulayness and unmixedness with respect to $P$,  $Q$ and $P+Q$}

In \cite{AR2} we call $M$ to be relative Cohen--Macaulay with respect to $Q$ if $H^i_Q(M)=0$ for all $i \neq q$ with $q\geq 0$. In other words, $\grade(Q, M)=\cd(Q, M)$ where $\cd(Q, M)$ denote the cohomological dimension of $M$ with respect to $Q$.
We recall the following facts from \cite{AR2} which will be used in the sequel.
\begin{eqnarray}
\label{3}
\cd(P,M)=\dim M/QM  \quad \text {and} \quad \cd(Q,M)=\dim M/PM.
\end{eqnarray}
It is natural to ask the following question:
\begin{Question}
Let $(R, \mm)$ be a Noetherian local ring, $I$ and $J$  two ideals of $R$ such that $I+J=\mm$ and $M$ a finitely generated $R$-module. If $M$ is relative Cohen--Macaulay with respect to $I$, i.e., $\grade(I,M)=\cd(I,M)$ and relative Cohen--Macaulay with respect to $J$, i.e., $\grade(J,M)=\cd(J,M)$.  Is $M$ itself Cohen--Macaulay?
\end{Question}
In the following, we give several examples which shows that the question is not the case in general for graded, local and bigraded cases.

\begin{Example}{\em
Consider the standard graded polynomial ring $S=K[x_1,\dots, x_{2n}]$ with $n\geq 1$ and $\deg x_i=1$ for all $i$.
Set $P=(x_1, \dots, x_n)$,  $Q=(x_{n+1}, \dots, x_{2n})$ and $\mm=(x_1, \dots, x_{2n})$ the unique graded maximal ideal of $S$. Set $R=S\oplus S/\pp$ where $\pp=(x_1+x_{n+1},x_2+x_{n+2},  \dots, x_n+x_{2n})$. One has that $S/\pp$ is Cohen--Macaulay
$S$-module of dimension $n$,  $\depth R=n$ and $\dim R= 2n$. On the other hand, $\grade(P,R)=\cd(P,R)=\grade(Q,R)=\cd(Q,R)=n$.
Thus $R$ is relative Cohen-Macaulay with respect to $P$ and $Q$, but $R$ itself is not Cohen--Macaulay.
Localizing $R$ at the maximal ideal $\mm$ and note that for any graded ideal $I$
of $S$  we have $\grade(I, R)=\grade(I_{\mm}, R_{\mm})$,
$\cd(I, R)=\cd(I_{\mm}, R_{\mm})$, $\depth_S R=\depth_{S_{\mm}} R_{\mm}$
and $\dim_S R=\dim_{S_{\mm}} R_{\mm}$. Now one easily deduces that the question is not the case in the local case  too. }
\end{Example}
\begin{Example}{\em
Let $n\geq 2$, and let $S=K[x_1, \dots, x_n,y_1, \dots, y_n]$ be the standard bigraded polynomial ring
 with $\deg x_i=(1,0)$ and $\deg y_i=(0,1)$ for $i= 1, \dots, n$. Set  $I=\bigcap_{i=1}^n \pp_i$ where $\pp_i=(x_i,y_i)$ for $i= 1, \dots, n$ and $R=S/I$. Let $\mm$ be the unique graded maximal ideal of $S$. From the exact sequence $0\rightarrow S/I \rightarrow \bigoplus_{i=1}^n S/\pp_i \rightarrow S/{\mm}\rightarrow 0$ we have the exact sequence
\[
\rightarrow H^j_Q(S/I) \rightarrow \bigoplus_{i=1}^n H^j_Q(S/\pp_i)\rightarrow H^j_Q(S/{\mm}) \rightarrow H^{j+1}_Q(S/I) \rightarrow.
\]
Note that $H^0_Q(S/{\mm})=S/{\mm}$ and $H^j_Q(S/\pp_i)=0$ for $j \neq n-1$ and all $i$. It follows that  $\grade(Q, R)=1$ and $\cd(Q,R)=n-1$. By a similar argument, applying the functor $H^i_P(-)$ to the above short exact sequence one obtains $\grade(P,R)=1$ and $\cd(P,R)=n-1$. Therefore $R$ is relative Cohen--Macaulay with respect to $P$ or $Q$ if and only if $n=2$. On the other hand, one has $\depth R=n-1$ and $\dim R=2(n-1)$. Thus if $n=2$, then $R$ is relative Cohen--Macaulay with respect to $P$ and $Q$, but not Cohen--Macaulay.}
\end{Example}
In the following we give two special cases in which the question holds. We recall the following theorem from \cite{AR2}
\begin{Theorem}
\label{main1}
Let $M$ be a finitely generated bigraded $S$-module which is relative Cohen--Macaulay with respect
to $Q$ and $|K|=\infty$. Then we have $\cd(Q ,M) +\cd(P,M) =\dim M$.
\end{Theorem}
\begin{Proposition}
\label{ab}
Let $M$ be a finitely generated bigraded $S$-module with $\cd(P,M)=p$ and  $\cd(Q,M)=q$ and let $|K|=\infty$. The following statements hold:
\begin{itemize}
\item[{(a)}] if  $M$ is  relative Cohen--Macaulay with respect to $P$ and $Q$ with $p=0$ or $p=\dim M$ and $q\geq 0$. Then $M$ is Cohen--Macaulay.
\item[{(b)}] if $M=M_1\tensor_K M_2$ where $M_1$ is  finitely generated graded $K[x]$-module and $M_2$ is  finitely generated graded $K[y]$-module. If $M$ is relative Cohen--Macaulay with respect to $P$ and $Q$, then $M$ is Cohen--Macaulay.
\end{itemize}
\end{Proposition}
\begin{proof}
In order to proof (a) we consider the spectral sequence $H^i_Q\big (H^j_{P}(M)\big )\underset{i} \Longrightarrow H^{i+j}_{\mm}(M)$ where $\mm=P+Q $. As $H^j_{P}(M)=0$ for all $j\neq 0$, then the above spectral sequence degenerates and one obtains  for all $i $ the following isomorphism of bigraded $S$-modules, $H^i_{Q}\big (H^0_{P}(M)\big) \iso H^{i}_{\mm}(M)$.
Using the fact that $\cd(P,M)=0$ if and only if $H^0_{Q}(M)=M$,  we therefore have  $H^i_{Q}(M) \iso H^{i}_{\mm}(M)$.
Since $H^{i}_{Q}(M) = 0$  for all  $i \neq  q $,   it follows that $H^{i}_{\mm}(M)= 0$ for all $i\neq q$ and so $M$ is Cohen--Macaulay.  Now let $p=\dim M$. By Theorem \ref{main1}, we have $q=0$ and then by a similar proof as above $M$ is Cohen--Macaulay.

In order to proof (b) we note that $H^i_P(M)\iso H^i_P(M_1)\tensor_KM_2$ for all $i$, see the proof \cite [Proposition 1.5]{AR2}. Since $M$ is relative Cohen--Macaulay with respect to $P$, it follows that $M_1$ is Cohen--Macaulay of dimension $p$. In fact, since $H^p_P(M)\neq 0$, it follows that $H^p_P(M_1)\neq 0$ and so  $p\leq \dim_{K[x]} M_1$. If  $p < \dim_{K[x]} M_1$, then $0=H^{\dim M_1}_P(M)\iso H^{\dim M_1}_P(M_1)\tensor_KM_2$. As  $M_2$ is  finitely generated faithful $K$-module, Grusen's theorem implies that $H^{\dim M_1}_P(M_1)=0$, a contradiction. Therefore $\dim_{K[x]} M_1=p$. By a similar argument as above  we have $\depth_{K[x]} M_1=p$. Similarly, from the isomorphism $H^i_Q(M)\iso M_1\tensor_K H^i_Q(M_2)$ for all $i$ we have  $\depth_{K[y]} M_2=\dim_{K[y]} M_2=q$. By \cite[Corollary 2.3]{STY} we have $\depth M=\dim M=p+q$, as desired.
\end{proof}
\begin{Remark}{\em
   By Proposition \ref{ab}(a), we deduce that Question 1.1 has positive answer while Example 1.3 shows that the question has negative answer when $\dim M=2.$ }
\end{Remark}
We recall the following known facts which will be used in the rest of paper:

Let $ 0 \rightarrow M' \rightarrow M \rightarrow M'' \rightarrow 0$  be an exact sequence of $S$-modules with $M$ finitely generated, then
\begin{eqnarray}
\label{1}
\cd(Q,M)=\max\{\cd(Q, M'), \cd(Q,M'')\},
\end{eqnarray}
Let $\Min M$ denote the minimal elements of $\Supp M$, then
\begin{eqnarray}
\label{2}
  \cd(Q,M)  =  \max \{\cd (Q, S/{\pp}): \pp \in \Ass(M)\}.
 \end{eqnarray}
Note also that
\begin{eqnarray*}
\cd(Q;M) &= & \max \{ \cd(Q,  S/\pp) : \pp \in \Supp(M)\}\\
        &= & \max \{ \cd(Q,  S/\pp) : \pp \in \Min(M)\}.
\end{eqnarray*}
\begin{Proposition}
\label{ass}
Let $M$ be a finitely generated bigraded $S$-module with $|K|=\infty$, then we have
\[
\grade(Q,M)\leq \cd(Q, S/{\pp}) \quad \text {for all} \quad \pp \in \Ass(M).
\]
\end{Proposition}
\begin{proof}
Here we follow the proof of \cite[Proposition 1.2.13]{BH}. Let $\pp \in \Ass M$. We proceed by induction on $\grade (Q,M)$. The claim is obvious if $\grade(Q,M)=0$. Now let $\grade(Q,M)=k>0$ and suppose inductively that the result has been proved for all finitely generated bigraded $S$-module $N$ such that  $\grade(Q,N)<k$.  We want to prove it for $M$. Since $\grade(Q, M)>0$, by \cite[Lemma 3.4]{AR2} there exists a bihomogeneous $M$-regular element $y\in Q$  which does not belong to any associated prime ideal of $M$ and not to any minimal prime ideal of $\Supp(M/PM)$ such that $\cd(Q, M/yM)=\cd(Q, M)-1$ and of course $\grade(Q, M/yM)=\grade(Q, M)-1$.  As in the proof of \cite[Proposition 1.2.13]{BH} we see that $\pp$ consists of zero divisors of $M/yM$.
Thus $\pp \subseteq \qq$ for some $\qq \in \Ass(M/yM)$. Since $y$ is $M$-regular, it follows that $y \not \in \pp$ while $y \in \qq$ and so $\pp \neq\qq$.  Note also that, as $y$ is $M$-regular and $\pp \in \Ass(M)$, we have that  $y$ is  $S/\pp$-regular and so $\grade(Q, S/\pp)>0$. Hence $\cd(Q, S/\pp)=\dim S/(P+\pp)>0$ by (1).   We claim that the  element $y$ may be chosen to avoid all the minimal prime ideal of  $\Supp(S/(P+\pp))$, too. Let  $\{\qq_1, \dots, \qq_r\}$ be the minimal prime ideals of  $\Supp(S/(P+\pp))$.  By  \cite[Lemma 3.3]{AR2} it suffices to show that  $Q\not \subseteq \qq_i$ for $i=1, \dots, r$. Suppose $Q \subseteq \qq_i$ for some $i$ where $i=1, \dots, r$. Since $P+\pp \subseteq \qq_i$, it follows  that $\qq_i=P+Q=\mm$, and hence $\dim S/(P+\pp)=\cd(Q, S/\pp)=0$, a contradiction. Using inductive hypothesis and the above observation we have
\begin{eqnarray*}
\grade(Q, M)-1 & = & \grade(Q, M/yM) \\ & \leq & \cd(Q, S/{\qq}) \\ & = &
\dim S/(P+\qq) \\ &<&  \dim S/(P+\pp)  =
\cd(Q, S/{\pp}),
\end{eqnarray*}
as desired.
\end{proof}
This in particular generalizes the following known results
\begin{Corollary}
Let $M$ be a finitely generated graded $K[y]$-module, then we have
\[
\depth M \leq \dim S/{\pp} \quad \text {for all} \quad \pp \in \Ass(M).
\]
In particular, $\depth M \leq \dim M$.
\end{Corollary}

\begin{Corollary}
Let $M$ be a finitely generated bigraded $S$-module, then we have
\[
\grade(Q,M)\leq\cd(Q,M).
\]

\end{Corollary}
\begin{proof}
The assertion follows from Proposition \ref{ass} and (3).
\end{proof}
\begin{Definition}{\em
Let $M$ be a finitely generated bigraded $S$-module. We call $M$ to be relative unmixed with respect to $Q$ if  $\cd(Q, M)=\cd(Q, S/{\pp})$ for all $\pp \in \Ass(M)$.}
\end{Definition}
In the following we observe that relative Cohen--Macaulay modules with respect to $Q$ are relative unmixed with respect to $Q$. In particular, all associated prime ideals of $M$ are minimal in $\Supp M/PM$.
\begin{Corollary}
\label{relative}
  Let $M$ be a finitely generated bigraded $S$-module which is relative Cohen--Macaulay with respect to $Q$, then $M$ is relative unmixed with respect to $Q$.
\end{Corollary}
\begin{proof}
By Proposition \ref{ass}, we have $\grade(Q, M)\leq \cd(Q, S/{\pp})$ for all $\pp \in \Ass(M)$. On the other hand, since $\pp \in \Ass(M) $, we have the monomorphism $S/{\pp}\rightarrow M$ which yields $\cd(Q, S/{\pp}) \leq \cd(Q, M)$ by (2). Thus the conclusion follows.
\end{proof}
\begin{Remark}{\em  Relative unmixed modules with respect to $Q$ need not to be relative Cohen--Macaulay with respect to $Q$.  We consider the hypersurface ring $R=S/fS$ where  $f\in S$ is a bihomogeneous polynomial of degree $(a,b)$ with  $a, b>0$  and $f$ is not monomial as well.  Note that $\Ass(R)=\{(f)\}$. One has $\grade(Q, R)=n-1$ and $\cd(Q, R)=n$. Thus $R$ is relative unmixed with respect to $Q$ but not relative Cohen--Macaulay with respect to $Q$ .}
\end{Remark}
The converse of Corollary \ref{relative} holds under the following additional assumption.
\begin{Proposition}
Let $M$ be a finitely generated bigraded $S$-module for which every quotient of $M$ is relative unmixed with respect to $Q$. Then $M$ is relative Cohen--Macaulay with respect to $Q$.
\end{Proposition}
\begin{proof}
We proceed by induction on  $q=\cd(Q, M).$ The claim is obvious for $q=0$. Assume $q>0$ and the result has been proved for all finitely generated bigraded $S$-module of cohomological dimension less than $q$. We may assume that $\grade(Q, M)>0$. Otherwise, $Q\subseteq \pp$ for some $\pp \in \Ass(M)$. Since $M$ is relative unmixed with respect to $Q$, we have $0<q=\cd(Q, S/{\pp})=\dim S/(P+\pp)\leq \dim S/(P+Q)=0$, a contradiction. By \cite[Lemma 3.4]{AR1} there exists an $M$-regular bihomogeneous element $y\in Q$ such that $\cd(Q, M/yM)=\cd(Q, M)-1$ as well as $\grade(Q, M/yM)=\grade(Q,M)-1$. Our assumption implies that $M/yM$ is relative unmixed with respect to $Q$ and hence our induction hypothesis says that $M/yM$ is relative Cohen--Macaulay with respect to $Q$. Therefore, $M$ is relative Cohen--Macaulay with respect to $Q$, as desired.
\end{proof}


The following question arises from Question 1.1:
\begin{Question}
\label{mix}
Let $(R, \mm)$ be a Noetherian local ring, $I$ and $J$  two ideals of $R$ such that $I+J=\mm$ and $M$ a finitely generated $R$-module. If $M$ is relative Cohen--Macaulay with respect to $I$ and relative unmixed with respect to $J$.  Is $M$ itself unmixed?
\end{Question}
\begin{Remark}{\em
In Example 1.2, we note that $\pp$ is  the only associated prime $S/\pp$ and so $\Ass(R)=\{\pp, (0)\}$. We have $\dim S/\pp=n< \dim R=2n$ while $R$ is relative Cohen--Macaulay with respect to $P$ and $Q$. Therefore the question  does not hold for $\dim M=2$.}
\end{Remark}
In the following we give several cases in which the Question \ref{mix} holds.
\begin{Proposition}
Let $|K|=\infty$ and let $M$ be relative Cohen--Macaulay  with respect to $P$ with  $\cd(P,M)=p \leq 1$ or $p=\dim M$ and relative unmixed with respect to $Q$ with $\cd(Q,M)=q\geq 0$. Then $M$ is unmixed.
 \end{Proposition}
 \begin{proof}
 Let $\pp \in \Ass M$. We first assume that $p=0$ and so $\cd(P,M)=\cd(P,S/\pp)=0$. Hence Theorem \ref{main1} yields  $\cd(Q,M)=\dim M$ and $\cd(Q, S/\pp)=\dim S/\pp$. Therefore relative unmixedness of $M$ with respect to $Q$ results that  $M$ is unmixed. Now let $p=1$ and so $\cd(P,M)=\cd(P,S/\pp)=1$. We claim that $S/\pp$ is relative Cohen--Macaulay with respect to $P$. Assume $\grade(P,S/\pp)=0$. The exact sequence $0 \longrightarrow S/\pp \longrightarrow M$ yields the exact sequence $0 \longrightarrow H^0_P(S/\pp) \longrightarrow H^0_P(M)$ and hence $\grade(P,M)=0$, a contradiction. By Theorem \ref{main1} we have
 \[
 \dim M=\cd(P,M)+\cd(Q,M)=\cd(P,S/\pp)+\cd(Q,S/\pp)=\dim S/\pp.
 \]
 The last equality follows again from Theorem \ref{main1}. Finally, we assume that  $p=\dim M$. Theorem \ref{main1} yields  $q=0$ and hence by a similar argument as the first part,  $M$ is unmixed.
 \end{proof}
 \begin{Corollary}
 \label{dim3}
 Let $\dim M\leq 3$ and $|K|=\infty$. If $M$ is relative Cohen--Macaulay with respect to $P$ and relative unmixed with respect to $Q$. Then $M$ is unmixed.
 \end{Corollary}

 \begin{Proposition}
Let $M_1$ and $M_2$ be two non zero finitely generated graded module over $K[x]$ and $K[y]$, respectively and let $|K|=\infty$. Set $M=M_1\tensor_K M_2$ and assume that $K[x]/{\pp_1}\tensor_K K[y]/{\pp_2}$ is an integral domain for all $\pp_1\in \Ass M_1$ and $\pp_2 \in \Ass M_2$. If $M$ is relative Cohen--Macaulay with respect to $P$ and relative unmixed with respect to $Q$, then $M$ is unmixed.
 \end{Proposition}
 \begin{proof}
 Let $\pp \in \Ass(M)$. Note that
 \[
 \Ass_S(M)=\bigcup_{\pp_1\in \Ass_{K[x]}(M_1)} \bigcup_{\pp_2\in \Ass_{K[y]}(M_2)} \Ass_S(K[x]/{\pp_1}\tensor_K K[y]/{\pp_2}),
 \]
 see \cite[Corollary 3.7]{STY}. Thus there exist $\pp_1\in \Ass_{K[x]}(M_1)$ and $\pp_2\in \Ass_{K[y]}(M_2)$ such that $\pp \in \Ass_S(K[x]/{\pp_1}\tensor_K K[y]/{\pp_2})=\Ass(S/\pp_1S+\pp_2S)$. By our assumption $S/(\pp_1S+\pp_2S)$ is an integral domain and so $\Ass(S/\pp_1S+\pp_2S)=\{\pp_1S+\pp_2S \}$. Hence $\pp=\pp_1+\pp_2$. Since $M$ is relative Cohen--Macaulay with respect to $P$, it follows that $M$ is relative unmixed with respect to $P$ and so we have
 \[
 \cd(P,M)=\cd(P,S/{\pp})=\dim S/(Q+\pp)=\dim S/(Q+\pp_1)=\dim K[x]/{\pp_1}.
 \]
 On the other hand, since $M$ is relative unmixed with respect to $Q$, we have
 \[
  \cd(Q,M)=\cd(Q,S/{\pp})=\dim S/(P+\pp)=\dim S/(P+\pp_2)=\dim K[y]/{\pp_2}.
 \]
Thus by Theorem \ref{main1} and \cite[Corollary 2.3]{STY}, we have
\[
\dim M=\cd(P,M)+\cd(Q,M)=\dim K[x]/{\pp_1}+\dim K[y]/{\pp_2}=\dim S/\pp,
\]
as desired.
 \end{proof}
\begin{Proposition}
 Let  $M$ be a finitely generated bigraded $S$-module such that every cyclic submodule of $M$ is pure. Let $|K|=\infty$ and assume $M$ is relative Cohen--Macaulay with respect to $P$ with $\cd(P, M)=p$ and relative unmixed with respect to $Q$ with $\cd(Q, M)=q$. Then, $M$ is unmixed.
  \end{Proposition}
  \begin{proof}
Let $\pp \in \Ass(M)$. We claim that $S/{\pp}$ is relative Cohen--Macaulay with respect to $P$. Let $f_1, \dots, f_p$ be a maximal $M$-sequence in $P$. Since $S/{\pp}$ is a cyclic submodule of $M$, the exact sequence $0 \longrightarrow S/{\pp}\longrightarrow M$ yields the exact sequence  $0 \longrightarrow S/\big({\pp}+(f_1, \dots, f_p)\big)\longrightarrow M/(f_1, \dots, f_p)M$. Since  $f_i\not \in Z(M/(f_1, \dots, f_{i-1})M)$ for all $i=1, \dots, p$, it follows that $f_i\not \in Z(S/\big({\pp}+(f_1, \dots, f_{i-1}))\big)$ for all $i=1, \dots, p$. Thus $f_1, \dots, f_p$ is an $S/{\pp}$-sequence in $P$ which may not be maximal. Hence  $\grade(P,S/{\pp})\geq p$. On the other hand, relative Cohen--Macaulayness of $M$ with respect to $P$ results that $M$ is relative unmixed with respect to $P$ and we have $\cd(P,M)=\cd(P,S/{\pp})=p$. Thus $\grade(P,S/{\pp})\geq p=\cd(P,S/{\pp})$. We conclude that  $\grade(P,S/{\pp})=\cd(P,S/{\pp})=p$ and so $S/{\pp}$ is relative Cohen--Macaulay with respect to $P$. Using Theorem \ref{main1} we have
  \[
  \dim M=\cd(P,M)+\cd(Q,M)=\cd(P,S/\pp)+\cd(Q,S/\pp)=\dim S/\pp,
  \]
  as desired.
  \end{proof}

\begin{Proposition}
 Let $I \subseteq S$ be a monomial ideal and set $R=S/I$ with $|K|=\infty$. Assume that $R$ is relative Cohen--Macaulay with respect to $P$ with $\cd(P, R)=p$ and relative unmixed with respect to $Q$ with $\cd(Q, R)=q$. Then $R$ is unmixed.
 \end{Proposition}
 \begin{proof}
 Let $\pp \in \Ass(R)$. By Corollary 1.4,  we have $\cd(P,R)=\cd(P, S/{\pp})=p$ and by our assumption  $\cd(Q, R)=\cd(Q, S/{\pp})=q$.  Note that, the associated prime ideals of a monomial ideal are monomial prime ideals, see \cite[Corollary 1.3.9]{HH}.  The equality $\dim S/(Q+{\pp})=p$ guaranties the existence $x_{i_{p+1}},\dots, x_{i_{m}}\in \pp$ for which  $x_{i_{1}}, \dots,  x_{i_{p}} \not \in \pp$ where $x_{i_{1}}, \dots, x_{i_{m}} \in \{x_1,\dots, x_m\}$ for all $i$. On the other hand, the equality $\dim S/(P+\pp)=q$ guaranties the existence $y_{j_{q+1}},\dots, y_{j_{n}} \in \pp$ for which $y_{i_{1}}, \dots, y_{i_{q}} \not \in \pp$ where $y_{j_{1}}, \dots, y_{i_{n}} \in \{y_1,\dots, y_n\}$ for all $i$.   Thus, we conclude that  $\pp=(x_{i_{p+1}},\dots, x_{i_{m}}, y_{j_{q+1}},\dots, y_{j_{n}})$. Therefore $\dim S/\pp=p+q=\dim R$ which follows from Theorem \ref{main1}.
 \end{proof}
.
\begin{Remark}{\em
Let $M$ be a relative Cohen--Macaulay with respect to $Q$ with $\cd(Q,M)=q$ and relative unmixed with respect to $P$ with $\cd(P,M)=p$ for which  $M$ is unmixed. Then all the associated prime ideals of $M$ have the same height, namely $n+m-(p+q).$ }
\end{Remark}
In Corollary \ref{dim3}, we observed that the Question \ref{mix} holds for $\dim M\leq 3$. We end this section with the following question:
\begin{Question}{\em
Let $M$ be a finitely generated bigraded $S$-module of dimension 4 which is relative Cohen--Macaulay with respect to $P$ and $Q$. Is the module $M$ unmixed? }
\end{Question}

\section{The krull-dimension of the graded components of local cohomology of relative cohen--macaulay modules}

In this section we describe explicitly the krull-dimension of the graded components of local cohomology of relative Cohen--Macaulay modules. As a first result we have the following
\begin{Proposition}
\label{dim}
Let $M$ be a finitely generated bigraded $S$-module with $\cd(P,M)=p$, $\cd(Q,M)=q$  and $|K|=\infty$. The following statements hold:
\begin{itemize}
\item[{(a)}] if $M$ is relative Cohen--Macaulay with respect to $Q$, then  $\dim_S H^q_{Q}(M)=p$,
\item [(b)] if $M$ is relative Cohen--Macaulay with respect to $P$, then  $\dim_S H^p_{P}(M)=q$.
\end{itemize}
\end{Proposition}
\begin{proof}
In order to prove (a) we note that $\Supp H^q_{Q}(M)\subseteq \Supp M/QM$. Thus we have  $\dim H^q_{Q}(M)\leq \dim M/QM=\cd(P,M)=p$. Since $M$ is relative Cohen--Macaulay with respect to $Q$,  from the spectral sequence $H^i_P\big (H^j_{Q}(M)\big )\underset{i} \Longrightarrow H^{i+j}_{\mm}(M)$ we get the following isomorphisms of bigraded $S$-modules $H^i_P\big (H^q_{Q}(M)\big )\iso H^{i+q}_{\mm}(M)$ for all $i$. By  Theorem \ref{main1}  we have $p+q=\dim M$ which yields $H^p_P\big (H^q_{Q}(M)\big) \neq 0$ and $H^i_P\big (H^q_{Q}(M)\big)=0$ for $i>p$. Thus, we conclude that $\cd(P,H^q_{Q}(M))=p$,  which is always less than or equal $\dim H^q_{Q}(M)$. Therefore $\dim_S H^q_{Q}(M)=p$. Part (b) is proved in the same way.
\end{proof}
Let $M$ be a finitely generated bigraded $S$-module. Recall the finiteness dimension of $M$ relative to $Q$ by:
\[
f_Q(M)=\inf \{ i \in \NN: H^i_Q(M) \; \text {is not finitely generated} \; \}.
\]
\begin{Proposition}
\label{finitedim}
Let $M$ be a finitely generated bigraded $S$-module with $\cd(P,M)=p$, $\cd(Q,M)=q$ and  $p+q=\dim M$. Then the following statements hold:
\begin{itemize}
\item[{(a)}] if $f_Q(M)=\cd(Q, M)=q$, then  $\dim_{K[x]} H^q_{Q}(M)_j=p$ for $j\ll 0$.
\item [(b)] if $f_P(M)=\cd(P, M)=p$, then  $\dim_{K[y]} H^p_{P}(M)_j=q$ for $j\ll 0$.
\end{itemize}
\end{Proposition}
\begin{proof}
For the proof (a), we consider the spectral sequence $H^i_P\big (H^k_{Q}(M)\big )_{(*,j)}\underset{i} \Longrightarrow H^{i+k}_{\mm}(M)_{(*,j)}$. Observe that $H^i_P\big (H^k_{Q}(M)\big )_{(*,j)} = H^i_{P_0}\big (H^k_{Q}(M)_{(*,j)}\big)$ where $P_0$ is the graded maximal ideal of $K[x]$. This equality follows from the definition of local cohomology using the \v{C}ech complex. Note that  $H^k_{Q}(M)_j=0 $ for all $k<\cd(Q, M)=q$ and $j\ll 0$.  Thus the spectral sequence degenerates and one obtains  for all $i$ and $j \ll 0$ the following isomorphisms of bigraded $K[x]$-modules
 $H^i_{P_0}\big (H^q_{Q}(M)_{(*,j)}\big)\iso H^{i+q}_{\mm}(M)_{(*,j)}$. Since $ H^{p+q}_{\mm}(M)$ is a non zero Artinian $S$-module which is not finitely generated, it follows that $H^{p+q}_{\mm}(M)_j \neq 0$ for $j\ll 0$. Thus $H^p_{P_0}\big (H^q_{Q}(M)_{j}\big) \neq 0$ for $j\ll 0$ while $H^i_{P_0}\big (H^q_{Q}(M)_{j}\big)=0$ for $i>p$. Therefore $\dim_{K[x]} H^q_{Q}(M)_j=p$ for $j\ll 0$, as desired. Part (b) is proved in the same way.
\end{proof}

\begin{Corollary}
\label{artin}
Let $M$ be a finitely generated bigraded $S$-module with  $f_Q(M)=\cd(Q, M)=q$,   $p+q=\dim M$ and $|K|=\infty$. Then $H^q_{Q}(M)$ is an Artinian $S$-module if and only if $q =\dim(M)$.
\end{Corollary}
\begin{proof}
Assume that $H^q_{Q}(M)$ is an Artinian $S$-module. Then,  one has that  $H^q_{Q}(M)_j$ is an Artinian $K[x]$-module for all $j$. Therefore  $\dim_{K[x]}H^q_{Q}(M)_j=0$ for all $j$. On the other hand, in view of Proposition \ref{finitedim}(a) we have that $\dim_{K[x]}H^q_{Q}(M)_j=\cd(P, M)$ for $j\ll 0$. Thus, we deduce that   $\cd(P, M)=0$ and hence by Theorem \ref{main1} we have  $q=\dim M$. The converse is a well known fact.
\end{proof}
\begin{Corollary}
\label{artcm}
Let $M$ be a finitely generated bigraded $S$-module with $|K|=\infty$. The following statements hold:
\begin{itemize}
\item[{(a)}] if $M$ is relative Cohen--Macaulay with respect to $Q$, then  $\dim_{K[x]} H^q_{Q}(M)_j=p$ for $j\ll 0$. Moreover, $H^q_{Q}(M)$ is an Artinian $S$-module if and only if $q =\dim(M)$.
\item [(b)] if $M$ is relative Cohen--Macaulay with respect to $P$, then  $\dim_{K[y]} H^p_{P}(M)_j=q$ for $j\ll 0$. Moreover, $H^p_{P}(M)$ is an Artinian $S$-module if and only if $p =\dim(M)$.
\end{itemize}
\end{Corollary}
\begin{proof}
The assertion follows from  Proposition \ref{finitedim}, Theorem \ref{main1} and  Corollary \ref{artin}.
\end{proof}
Recall the $Q$-finiteness dimension $f_{\mm}^Q(M)$ of $M$ relative to $\mm$ by
\[
f_{\mm}^Q(M)=\inf \{i\in \NN_0:  Q \not \subseteq   \rad (0: H^i_{\mm}(M)) \}.
\]
In view of \cite [Proposition 2.3]{M}  one has
\[
f_{\mm}^Q(M)=\sup \{i\in \NN_0: H^k_{\mm}(M)_j= 0 \;  \text {for all}  \; k < i  \; \text {and all}\; j\ll 0 \; \}.
\]

\begin{Proposition}
\label{dim}
Let $M$ be a finitely generated bigraded $S$-module with $\cd(P,M)=p$, $\cd(Q,M)=q$ and  $p+q=\dim M$. The following statements hold:
\begin{itemize}
\item[{(a)}] if $M$ is generalized Cohen--Macaulay with $f_Q(M)=\cd(Q, M)=q$, then   $\depth_{K[x]} H^q_{Q}(M)_j=p$ for $j\ll 0$.
\item [(b)] if $M$ is generalized Cohen--Macaulay with $f_P(M)=\cd(P, M)=p$, then $\depth_{K[y]} H^p_{P}(M)_j=q$ for $j\ll 0$.
\end{itemize}
\end{Proposition}
\begin{proof}
 For the proof (a), since $M$ is generalized Cohen--Macaulay, it follows that $f_{\mm}^Q(M)=\dim(M)=p+q$.  By \cite[Theorem 2.3]{HJZ} we have $\grade(P_0, H^q_{Q}(M)_j)=f_{\mm}^Q(M)-\cd(Q, M)$ for $j \ll 0$. This yields the desired claim. Part (b) is proved in the same way.
\end{proof}
\begin{Corollary}
Let $M$ be a finitely generated bigraded generalized Cohen--Macaulay $S$-module with $f_Q(M)=\cd(Q, M)=q$ and $p+q=\dim M$. Then  $H^q_{Q}(M)_j$ is Cohen--Macaulay $K[x]$-module of dimension $p$  for $j\ll 0$ and $\projdim_{K[x]} H^q_{Q}(M)_j=n-p$ for $j\ll 0$ .
\end{Corollary}


\section{Finiteness properties of local cohomology of an hypersurface ring }
 This is a well-known fact that the top local cohomology modules are almost never finitely generated. Let $M$ be relative Cohen--Macaulay with respect to $Q$ with $\cd(Q, M)=q$. Thus $H^{q}_{Q}(M)$ is
not finitely generated for $q>0$. In Corollary \ref{artcm} we observed that $H^{q}_{Q}(M)$ is not
artinian as well, unless the ordinary known case $q=\dim M$. We consider the hypersurface ring $R=S/fS$
where  $f \in S$ is a bihomogeneous form of degree $(a, b)$. This ring has only two nonvanishing local
cohomology with respect to $P$ or $Q$ which is close to relative Cohen--Macaulay modules. In the following ,
we first observe that $H^{n-1}_{Q}(R)$ is not finitely generated, too for $n\geq 2$ and obtain some results on
 Artinianness of local cohomology of $R$.
\begin {Proposition}
Let $R=S/fS$ be a  hypersurface ring. Then $H^{n-1}_{Q}(R)$ is not finitely generated for $n\geq 2$.
\end{Proposition}

\begin {proof}
\label{surface}
 The exact sequence  $0\rightarrow S(-a,-b)\stackrel f \rightarrow S\rightarrow S/fS\rightarrow 0$, induces the following exact sequence of $S$-modules
\[
0 \rightarrow  H_{Q}^{n-1}(R) \rightarrow H_{Q}^{n}(S)(-a, -b) \stackrel f \rightarrow  H_{Q}^{n}(S) \rightarrow H_{Q}^n(R) \rightarrow 0.
\]
Moreover,  $H^i_{Q}(R)= 0$ for all $i<n-1$.  Let $F$ be the quotient field of $K[x]$. Then
\[
F \tensor_{K[x]} S=F[y_1,\ldots,y_n]=:T.
\]
Let $T_+$ be the graded maximal ideal of $T$.  By the graded flat base change theorem, we have
\[
F \tensor_{K[x]}H_{Q}^i(R)\iso H_{T_+}^i(F\tensor_{K[x]} R )\quad \text {for all} \quad  i.
\]
Since $ F \tensor_{K[x]} R = T/fT$ and $\dim T/fT = n-1$, it follows that
\[
H_{T_+}^i(T/fT)=0  \quad \text {for all} \quad  i \neq n-1.
\]
Note that $H_{T_+}^{n-1}(T/fT)$ is an Artinian $T$-module which is not finitely generated. Thus  $H_{T_+}^{n-1}(T/fT)_j \neq 0$ for all $j\ll 0$ and $n\geq 2$, and hence $H_{Q}^{n-1}(R)_j \neq 0$ for all $j\ll 0$ and $n\geq 2$. Therefore $H_{Q}^{n-1}(R)$ is not finitely generated for $n\geq 2$, as desired.
\end {proof}
For bihomogeneous element $f \in S$, we denote by $c(f)$  the ideal of $K[x]$ generated by all the coefficients of $f$ and $P_0=(x_1, \dots, x_m)$ the graded maximal ideal of $K[x]$. A dual version of the above observation can be discussed as Artinianness of local cohomology of hypersurface rings.

\begin{Proposition}
Let $R=S/fS$ be a hypersurface ring. Then,
\begin{itemize}
\item[{(a)}]  if   $m\leq 1$, then $H_{Q}^{n}(R)$ is an Artinian $S$-module,
\item[(b)] let $m\geq 2$.  If $H_{Q}^{n}(R)$ is an Artinian $S$-module, then $c(f)$ is an $P_0$-primary ideal which does not form a system of parameters for $K[x]$.
\end{itemize}
\end{Proposition}
\begin{proof}

For the proof (a) if $m=0$, then $Q$ is the graded maximal ideal of $K[y]$ and we may write $f= \sum_
{\left|\beta\right|=b}c_{\beta}  y^\beta $ where $ c_{\beta}\in K$. Hence $R$ is  Cohen-Macaulay of dimension $n-1$ and so $H_{Q}^{n}(R)=0$ is Artinian. Let  $m=1$,   we need to show  $\Supp H_{Q}^{n}(R)\subseteq \{ \mm\}$ and $\Hom(S/{\mm}, H_{Q}^{n}(R))$ is a finitely generated $S$-module where $\mm=P+Q$ is the unique graded maximal ideal of $S$. As  $m=1$, we may write  $f=x^a \sum_
{\left|\beta\right|=b}c_{\beta}  y^\beta $ where $ c_{\beta}\in K$. Thus $c(f)=(x^a)$ is an $(x)$-primary ideal and hence by \cite[Corollary 2.6]{MV} we have  $\Supp H_{Q}^{n}(R)= \{ \mm\}$. Since $m=1$, by \cite [Theorem 1]{DT}  $H_{Q}^{n}(R)$ is $Q$-cofinite and so $\Hom(S/Q, H_{Q}^{n}(R))$ is finitely generated. Therefore $\Hom(S/{\mm}, H_{Q}^{n}(R))$ is finitely generated.

For the proof (b), as  $H_{Q}^{n}(R)$ is an Artinian $S$-module, we have $\Supp H_{Q}^{n}(R)\subseteq \{ \mm\}$. By \cite [Lemma 2.5]{MV} we have $\Supp H_{Q}^{n}(R)=\{ \qq \in \Spec S: c(f)+Q \subseteq \qq \}$. Thus the maximal ideal $\mm$ is the only minimal prime ideal $c(f)+Q$. It follows that $P_0$ is the only minimal prime ideal $c(f)$. Therefore $c(f)$ is an $P_0$-primary ideal. Since $\Hom(S/{\mm}, H_{Q}^{n}(R))$ is finitely generated, by \cite[Theorem 2.3]{MV},  $c(f)$ does not form a system of parameters for $K[x]$.
\end{proof}
In the following, we show that $H_{Q}^{n-1}(R)$ is Artinian if and only if  $Q$ is the graded maximal ideal and $\dim R=n-1$.
\begin{Proposition}
Let $R=S/fS$ be a hypersurface ring and $n\geq 1$. Then $H_{Q}^{n-1}(R)$ is an Artinian $S$-module if and only if $m=0$.
\end{Proposition}
\begin{proof}
The exact sequence  $0\rightarrow S(-a,-b)\stackrel f \rightarrow S\rightarrow S/fS\rightarrow 0$,
induces the following exact sequence of $S$-modules
\[
0 \rightarrow  H_{Q}^{n-1}(R) \rightarrow H_{Q}^{n}(S)(-a, -b) \stackrel f \rightarrow  H_{Q}^{n}(S) \rightarrow H_{Q}^n(R) \rightarrow 0.
\]
Note that $H_{Q}^{n-1}(R)=0 \underset {H_{Q}^{n}(S)} :f \supseteq 0 \underset {H_{Q}^{n}(S)} :Q$ and  $H_{Q}^{n}(S)$ is an $Q$-torsion $S$-module. By \cite[Theorem 7.1.2]{BS} we have $H_{Q}^{n-1}(R)$ is an Artinian $S$-module if and only if $H_{Q}^{n}(S)$ is an Artinian $S$-module. Hence by Corollary \ref{artcm} this is equivalent to saying that $m=0$.
\end{proof}
\section{generalization of Reisner's criterion for Cohen--Macaulay simplicial complexes}

As before, let $S=K[x_1, \dots, x_m, y_1, \dots, y_n]$ be the standard bigraded polynomial ring in $n+m$ variables over a field $K$ and $\Delta$ a simplicial complex on $[n+m]$. We assume that $\Delta$ has vertices $\{v_1, \dots, v_m, w_1, \dots, w_n\}$ where vertices $V=\{v_1, \dots, v_m\} $ and $W=\{w_1, \dots, w_n\}$ correspond to the variables of $x_1, \dots, x_m$ and $y_1, \dots,  y_n$, respectively. We denote by $\Delta_W$ the restriction of $\Delta$ on $W$ which is the subcomplex
\[
\Delta_W=\{ F \in \Delta: F \subseteq W \}.
\]
Let $F$ be a facet simplicial complex of $\Delta$ on $[n+m]$.  We denote by $\pp_F$ the prime ideal generated by all $x_i$ and $y_j$ such that $v_i, w_j \not \in F$.
\begin{Proposition}
\label{facet}
 Let $\Delta$ be a simplicial complex on $[n+m]$ and  $K[\Delta]=S/I_{\Delta}$ the Stanley-Reisner ring of  $\Delta$. Then
\[
\cd(Q, K[\Delta])=\dim \Delta_W+1.
\]
\end{Proposition}
\begin{proof}
Using primary decomposition of $I_{\Delta}=\bigcap_F \pp_F$ where the intersection is taken over all facets of $\Delta$, together with (1) and  (3)  we have
\begin{eqnarray*}
\cd(Q, K[\Delta]) & = & \max \{\cd(Q, S/{\qq}): \qq \in \Min (K[\Delta])\} \\ & = &
 \max \{\cd(Q, S/{\pp_{F}}):  F \; \text{is a facet of} \; \Delta   \}
\\ & = &
\max \{\dim  S/(P+{\pp_{F}}):   F\; \text{ is a facet of} \; \Delta   \}\\ & = &
\max \{\dim  K[y_1, \dots, y_n]/{\pp_{G}}:   G\; \text{ is a facet of} \; \Delta_W   \}\\ & = &
\max \{ |G|: G  \; \text {is a facet of} \; \Delta_W \}\\ & = &
\dim \Delta_W+1,
\end{eqnarray*}
as required.
\end{proof}
\begin{Corollary}
Let $\Delta$ be a simplicial complex on $[n]$,  then one has $\dim K[\Delta]=\dim \Delta +1.$
\end{Corollary}

We say that $\Delta$ is relative Cohen--Macaulay with respect to $Q$ over $K$ if $K[\Delta]$ is relative Cohen--Macaulay with respect to $Q$. We say that a simplicial complex $\Delta$ is pure if all facets have the same cardinality.
\begin{Corollary}
\label{pure}
Let $\Delta$ is relative Cohen--Macaulay with respect to $Q$, then $\Delta_W$ is a pure simplicial complex.
\end{Corollary}
\begin{proof}
The assertion is immediate from Corollary \ref{relative} and Proposition \ref{facet}.
\end{proof}
\begin{Corollary}
\label{dim0}
Let $\dim \Delta_W=0$, then $\Delta$ is relative Cohen--Macaulay with respect to $Q$.
\end{Corollary}
\begin{proof}
By Proposition \ref{facet} we have $\cd(Q, K[\Delta])=1$. Since $\dim \Delta_W=0$, it follows that the facets of $\Delta_W$ are the forms $F_i=(w_i)$ for $i=1, \dots, n$ and hence $\pp_{F_i}=(x_1, \dots, x_m, y_1, \dots, \hat{y_i}, \dots, y_n)$ where $y_i \not \in \pp_{F_i}$. Thus $Q \not \subseteq \qq$ for all $\qq \in \Ass (K[\Delta])$. Therefore $\grade(Q, K[\Delta])=1$, as required.
\end{proof}
Let $\Delta$ be a simplicial complex on [n]. For a face $F$ of $\Delta$, the link of $F$ in $\Delta$ is the subcomplex
\[
\link_{\Delta} F=\{ G\in \Delta: F \cup G \in \Delta, F \cap G=\emptyset \},
\]
and the star of $F$ in  $\Delta$ is the subcomplex
\[
\star_{\Delta} F=\{ G\in \Delta: F \cup G \in \Delta \}.
\]
 Note that if $\Delta$ be a pure simplicial complex, then for any $F \in \Delta$ we have $\dim \link_{\Delta} F= \dim \Delta- \left|F\right|$. We denote by $\widetilde{H}_{i}(\Delta ;K) $ the $i$th reduced homology group of $\Delta$ with coefficient in $K$, see Chapter 5 in \cite{BH} for details. We say that a simplicial complex $\Delta$ is
connected if there exists a sequence of facets $F = F_0,  \dots,  F_t = G$
such that $F_i\cap F_{i+1}\neq 0$ for $i=0, \dots, t-1 $.  One has,  $\Delta$  is connected if and only if  $\widetilde{H}_{0}(\Delta ;K)=0 $.  We set $\ZZ^m_-=\{a \in \ZZ^m: a_i\leq 0 \; \text{ for } \; i=1, \dots, m\}$ and $\ZZ^n_+=\{b \in \ZZ^n: b_i\geq 0 \; \text {for } \; i=1, \dots, n\}$.
We recall the following theorem from \cite[Theorem 1.3]{AR1}.
\begin{Theorem}
\label{Hochster} Let $I \subset S$ be a squarefree monomial ideal. Then the bigraded Hilbert series of the
local cohomology modules of $K[\Delta]=S/I$ with respect to the
$\ZZ^m \times \ZZ^n$-bigrading is given by
\begin{eqnarray*}
H_{H_{Q}^i(K[\Delta])}(\bold{s},\bold{t})& = & \sum_{{a\in \ZZ^m_+ , b\in \ZZ_-^n}}\dim_K H_{P}^i(K[\Delta])_{(a,b)}\bold{s}^a \bold{t}^b \\&=&\sum_{F\in \Delta_W}\sum_{G\subset V}\dim _K \widetilde{H}_{i-\left|F\right|-1}((\link F \cup G)_W;K)\prod _{v_i\in G}\frac{s_i}{1-s_i} \prod_{w_j\in F}\frac{t_j^{-1}}{1-t_j^{-1}}
\end{eqnarray*}
where $\bold{s}=(s_1, \dots, s_m)$,  $\bold{t}=(t_1, \dots, t_n)$,  $G=\Supp a$,  $F=\Supp b$ and $\Delta$ is the simplicial complex corresponding to the Stanley-Reisner ring $K[\Delta]$.
\end{Theorem}
Here we note that $(\link F \cup G)_W=\link_{\Delta_W}F$.
 As an immediate consequence we obtain
\begin{Corollary}
\label{reduced}
 We have $H_{Q}^i(K[\Delta])_{(a,b)}=0$ for all $i$ and for all $ b \in \ZZ^n$ for which $b_j> 0$ for some $j$,   or for all $ a \in \ZZ^m$ for which $a_i< 0$ for some $i$ and
\[
H_{Q}^i(K[\Delta])_{(a,b)}\iso \widetilde{H}_{i-\left|F\right|-1}((\link F \cup G)_W;K) \quad \text {for all} \quad a \in \ZZ^m_+ \quad \text {and all} \quad  b \in \ZZ^n_-,
\]
 where  $G=\Supp a$ and  $F=\Supp b$.
\end{Corollary}
As a main result of this section we have the following. Here we follow the proof \cite[Theorem 8.1.6]{HH}.
\begin{Theorem}
\label{main}
Let $\Delta$ be a simplicial complex over a field $K$. The following conditions are equivalent.
\begin{itemize}
\item[{(a)}] $\Delta$ is relative Cohen--Macaulay with respect to $Q$ with $\cd(Q,K[\Delta])=q$,
\item[(b)] $\widetilde{H}_{i}((\link F \cup G)_W;K)=0$ for all $F\in \Delta_W$, $G\subset V$ and all $i< \dim \link_{\Delta_W} F$.
\end{itemize}
\end{Theorem}
\begin{proof}
Note that $\dim \Delta_W=q-1$ by Proposition \ref{facet}. Let $\Delta$ be relative Cohen--Macaulay with respect to $Q$. This is equivalent to saying that  $H_{Q}^i(K[\Delta])=0$ for all $i \neq q$. Hence by Corollary \ref{reduced}, this is equivalent to saying that
\begin{eqnarray}
\label{3}
\widetilde{H}_{i-\left|F\right|-1}((\link F \cup G)_W;K)=0 \; \text { for all} \; F\in \Delta_W, G\subset V \; \text{ and all} \; i< q.
\end{eqnarray}
$(a)\Longrightarrow (b)$:  Since $\Delta$ is relative Cohen--Macaulay with respect to $Q$, by Corollary \ref{pure} it follows that  $\Delta_W$ is pure and hence $\dim \link_{\Delta_W} F=\dim \Delta_W-\left|F\right|=q-\left|F\right|-1$. Therefore (4) implies that $\widetilde{H}_{i}((\link F \cup G)_W;K)=0$ for all $F\in \Delta_W$, $G\subset V$ and all $i< \dim \link_{\Delta_W} F$.

$(b)\Longrightarrow (a)$: Let $F\in \Delta_W$, $G\subset V$  and  $H\in \link_{\Delta_W} F$. Set $ \Gamma=\link_{\Delta_W} F$. One has
\[
\link _{\Gamma}H= \link_{\Delta_W}(H \cup F)=\big (\link (H\cup F \cup G) \big)_W.
\]
Hence our assumption yields
\[
\widetilde{H}_{i}(\link _{\Gamma}H;K)=\widetilde{H}_{i}\big(\big (\link (H\cup F \cup G) \big)_W ;K\big)=0 \; \text { for  all } \;  i< \dim \link_{\Gamma} H.
\]
Thus, by induction on the $\dim \Delta_W$ we may assume that all proper links of $\Delta_W$ are Cohen--Macaulay over $K$. In particular, the link of each vertex of $\Delta_W$ is pure. Thus all facets containing a given vertex have the same dimension. Now, let $\dim \Delta_W=0$, by Corollary \ref{dim0},  $\Delta$ is relative Cohen--Macaulay with respect to $Q$. Thus we may assume that $\dim \Delta_W \geq 1.$ Since $\widetilde{H}_0(\Delta_W; K)=\widetilde{H}_0(\link_{\Delta_W}\emptyset ;K)=0$, it follows that $\Delta_W$ is connected.
Thus $\Delta_W $ is a pure simplicial complex and hence for any $F\in \Delta_W$, we have $\dim \link_{\Delta_W} F=q-\left|F\right|-1$. Thus our hypothesis implies (4) and so $\Delta$ is relative Cohen--Macaulay with respect to $Q$.
\end{proof}

As an immediate consequence we obtain the Reisner's criterion for Cohen--Macaulay simplicial complexes
\begin{Corollary}
\label{reisner}
Let $\Delta$ be a simplicial complex and $K$ a field. Then, $\Delta$ is Cohen--Macaulay over $K$ if and only if
$\widetilde{H}_{i}(\link F ;K)=0$ for all $F\in \Delta$ and all $i< \dim \link F$.
\end{Corollary}
\begin{proof}
In Theorem \ref{main} we assume that $m=0$, then $G=\emptyset$, $(\link F \cup G)_W=\link F$ , $\Delta_W=\Delta$ and $Q$ is the unique maximal ideal $\mm$ and  $\cd(Q, K[\Delta])=\dim K[\Delta]$.
\end{proof}
In the proof of the theorem we showed
\begin{Corollary}
Let $\Delta$ be relative Cohen–-Macaulay with respect to $Q$, then $\Delta_W$ is connected.
\end{Corollary}
\begin{Corollary}
Let $\Delta$ be a relative Cohen–-Macaulay complex with respect to $Q$ and F is a face of $\Delta_W$.
Then $\link_{\Delta_W} F$ is Cohen–-Macaulay.
\end{Corollary}
\begin{proof}
The assertion follows from the beginning of the proof Theorem \ref{main} $(b)\Longrightarrow (a)$ and Corollary \ref{reisner}.
\end{proof}

Let $I\subset S$ be a monomial ideal and $G(I)$ the unique minimal monomial system of generators of $I$. For a monomial $u \in S$ we may write $u=u_1u_2$ where $u_1=x_1^{c_1}\dots x_m^{c_m}$  and $u_2=y_1^{d_1}\dots y_n^{d_n}$.  We set $\nu_i(u_1)=c_i$ for $i=1,\ldots,m$ and $\nu_j(u_2)=d_j$ for $j=1,\ldots,n$. We also set $\sigma_i = \max\{\nu_i(u_1): u\in G(I)\}$
for $i=1,\ldots, m$  and $\rho_j = \max\{\nu_j(u_2)\: u\in G(I)\}$ for $j=1,\ldots, n$. For $b=(b_1, \dots, b_n) \in \ZZ^n$ we set $G_b=\{ j: 1\leq j \leq n, b_j<0 \}$ and let $a \in \ZZ^m_+$. We define the simplicial complex $\Delta_{(a,b)}(I)$ whose faces are the set $L-G_b$ with $G_b \subseteq L$ and such that $L$ satisfies the following conditions: for all $u\in G(I)$ there exists $j\notin L$  such that $\nu_j(u_2) > b_j\geq 0$,   or for at least one $i$, $\nu_i(u_1) > a_i\geq 0$.
We recall the following theorem from \cite[Theorem 2.4]{AR1}.
\begin{Theorem}
\label{Hochster 2} Let $I\subset S$ be a monomial ideal. Then the Hilbert series
of the local cohomology modules of $S/I$ with respect to the
$\ZZ^m \times \ZZ^n$-bigrading is given by
\[
H_{H_{Q}^i(S/I)}(\bold{s},\bold{t})
=\sum \sum \dim_K\tilde{H}_{i-\vert F\vert -1}(\Delta_{(a,b)}(I); K) {\bf s}^a {\bf t}^b,
\]
where the first sum runs over all $F\in \Delta_W$, $b\in\ZZ^n$ for which $G_b = F$ and $b_j\leq \rho_j-1$ for  $j=1,\ldots,n$, and the
second sum runs over all $a\in\ZZ^m$ for which $N_a = G$ and $a_i\geq
\sigma_i-1$ for  $i=1,\ldots,m$. Here $N_a=\Supp a$ and
 $\Delta$ is the
simplicial complex corresponding to the Stanley-Reisner
ideal $\sqrt{I}$.
\end{Theorem}
The precise expression of the Hilbert series is given in \cite{AR1}. As a first consequence we have
\begin{Corollary}
\label{zero}
we have $H_{Q}^i(S/I)_{(a,b)}=0$ for all $i$ and for all $b\in\ZZ^n$ for which $b_j> \rho_j-1$ for some $j$, or for all $ a \in \ZZ^m$ for which $a_i< \sigma_i-1 $ for some $i$  and
\[
H_{Q}^i(S/I)_{(a,b)} \iso \widetilde{H}_{i-\left|F\right|-1}(\Delta_{(a,b)}(I) ;K)
\]
for all  $b\in\ZZ^n$ with  $b_j\leq \rho_j-1$ for  $j=1,\ldots,n$, and $G_b =F$ and for all  $a \in \ZZ^m$ with $a_i\geq \sigma_i-1$ for  $i=1,\ldots,m$ and $N_a = G$.
\end{Corollary}
For a bigraded $S$-module $M$, we recall the  $a$-invariant of $M$ by
\[
a^i_Q(M)=\sup \{ \mu: H_{Q}^i(M)_{(*,\mu)}\neq 0 \},
\]
and so $\reg(M)=\underset{i}{\max} \{a^i_Q(M)+i: i\geq 0 \} $.
\begin{Corollary}
Suppose  $I\subseteq S$ be a monomial ideal such that $S/I$ is relative Cohen--Macaulay with respect to $Q$ with $\cd(Q,S/I)=q$. then
\[
\reg(S/I)\leq \sum_{j=1}^n \rho_j-n+q.
\]
\end{Corollary}
\begin{proof}
Note that for all $k , j
\in \ZZ$ we have
\[ 
H_{Q}^q(S/I)_{(k,j)}=\Dirsum_{{a \in \ZZ^m,\left|a\right|=k}\atop {b \in \ZZ^n,\left|b\right|=j}} H_{Q}^q(S/I)_{(a,b)},
\] 
where $\left|a\right|=\sum_{i=1}^m a_i$ for $a=(a_1,\dots,a_m)$ and $\left|b\right|=\sum_{i=1}^n b_i$ for $b=(b_1,\dots,b_n)$. By Corollary \ref{zero} we have that  $H_{Q}^q(S/I)_{(k,j)}=0$ for $k<\sum_{i=1}^m \sigma_i-m$  or $j>\sum_{j=1}^n \rho_j-n$. Thus we have
\[
H_{Q}^q(S/I)_j=\Dirsum_k H_{Q}^q(S/I)_{(k,j)}=0 \; \text { for } \; j>\sum_{j=1}^n \rho_j-n.
\]
Hence  $a^q_Q(S/I)\leq \sum_{j=1}^n \rho_j-n$ and so the conclusion follows.
\end{proof}
As a generalization of \cite[Corollary 2.3]{HTT} we have
\begin{Corollary}
\label{radical}
Let $I\subset S$ be a monomial ideal. Then for all $i$ we have the following isomorphisms of $K$-vector spaces
\[
H_{Q}^i(S/I)_{(a,b)}\iso H_{Q}^i(S/\sqrt{I})_{(a,b)},
\]
for all $a\in\ZZ^m_+$ and $b\in\ZZ^n_-.$ In particular, $\cd(Q,S/I)=\cd(Q,S/{\sqrt{I}})$.
\end{Corollary}
\begin{proof}
By a similar proof as \cite[Corollary 2.3]{HTT} one has  $\Delta_{(a,b)}(I)=\Delta_{(a,b)}(\sqrt{I})$. Thus Corollary \ref{zero} yields the desired isomorphism.
\end{proof}
Now we come to a general version of Theorem \ref{main} as follows:
\begin{Corollary}
Let $I \subseteq S$ be a monomial ideal and $\Delta$ the simplicial complex corresponding to  $\sqrt{I}$. The following conditions are equivalent.
\begin{itemize}
\item[{(a)}] $S/I$ is relative Cohen--Macaulay with respect to $Q$ with $\cd(Q,S/I)=q$,
\item[(b)] $\widetilde{H}_{i}((\link F \cup G)_W ;K)=0$ for all $F\in \Delta_W$, $G\subset V$ and all $i< \dim \link_{\Delta_W} F$.
\end{itemize}
\end{Corollary}
\begin{proof}
Note that
\[
\Delta_{(a,b)}(I)=\Delta_{(a,b)}(\sqrt{I})=\link_{\star N_a\cup H_b}G_b=\link_{\star N_a}G_b=(\link F \cup G)_W,
\]
see the remark after \cite[Theorem 2.4]{AR1} and also the proof \cite[Corollary 1]{T}. Now the assertion follows by applying Corollary \ref{zero} and Corollary \ref{radical} to Theorem \ref{main}.
\end{proof}

\end{document}